\documentclass[11pt]{amsart}
\usepackage{graphicx,tikz}
\usepackage{amssymb,amscd}
\usepackage{amsthm}
\usepackage{amsfonts,mathrsfs}
\usepackage{setspace}
\usepackage{enumerate}
\usetikzlibrary{arrows, automata,chains,fit,shapes}
\usepackage{subfig,float}
\usetikzlibrary{decorations.markings}

\tikzset{->-/.style={decoration={
  markings,
  mark=at position #1 with {\arrow{>}}},postaction={decorate}}}

\newenvironment{LC}{\noindent\color{green} LC:}{}

\newenvironment{MF}{\noindent\color{blue} \colorbox{blue}{\color{black} MF:}}{}

\newenvironment{ME}{\noindent\color{magenta} ME:} {}

\newcommand{\ra}{\rightarrow}

\DeclareMathOperator{\Aut}{Aut}

\newcommand\Z{\mathbb Z}
\newcommand\N{\mathbb N}

\newcommand\A{\mathbf A}

\newcommand\T{\mathbf T}

\newtheorem{thm}{Theorem}
\newtheorem{prop}[thm]{Proposition}
\newtheorem{lem}[thm]{Lemma}
\newtheorem{cor}[thm]{Corollary}

\newtheorem{conjecture}{Conjecture}[section]

\theoremstyle{definition}
\newtheorem{defn}[thm]{Definition}
\newtheorem{rmk}{Remark}
\newtheorem{question}[conjecture]{Question}
\newtheorem{example}{Example}

\begin{document}

\title{Applications of L systems to group theory}

\author{Laura Ciobanu}
\address{School of Mathematical and Computer Sciences,
 Heriot-Watt University, 
 Edinburgh EH14 4AS,
 Scotland}
\email{l.ciobanu@hw.ac.uk}

\author{Murray Elder}
\address{University of Technology Sydney, Ultimo NSW 2007, Australia}
\email{murrayelder@gmail.com}

\author{Michal Ferov}
\address{University of Technology Sydney, Ultimo NSW 2007, Australia}
\email{michal.ferov@gmail.com}

\begin{abstract}
L systems generalise context-free grammars by incorporating parallel rewriting, and generate languages such as EDT0L and ET0L that are strictly contained in the class of indexed  languages. In this paper we show that many of the languages naturally appearing in group theory, and that were known to be indexed or context-sensitive, are in fact ET0L and in many cases EDT0L. For instance, the language of primitives and bases in the free group on two generators, the Bridson-Gilman normal forms for the fundamental groups of 3-manifolds or orbifolds, and the co-word problem of Grigorchuk's group 
 can be generated by L systems. To complement the result on primitives in rank 2 free groups, we show that the language of primitives, and primitive sets, in free groups of rank higher than two is context-sensitive. We also show the existence of EDT0L and ET0L languages of intermediate growth.
\end{abstract}



\date{\today}
\maketitle

\section{Introduction}
In this paper we show that many of the context-sensitive or indexed languages arising in problems in group theory and combinatorics are in fact ET0L or even EDT0L. The merit of giving this new formal language characterisation is that ET0L and EDT0L languages are a strict subclass of the indexed ones, and also, the descriptions of these languages are simpler and more algebraic than those based on indexed grammars or nested-stack automata that were given in the literature for the sets considered here.

Both EDT0L and ET0L belong to the languages generated by L systems, which were introduced by Lindenmayer in the late 1960s in order to model the growth of various organisms.
The acronym {ET0L} (respectively {EDT0L}) refers to {\em {\rm E}xtended, {\rm T}able, {\rm 0} interaction, and {\rm L}indenmayer} (respectively D\emph{eterministic}). ET0L and EDT0L languages have only recently featured in group theory; their first prominent appearance was in the work of the first two authors with Volker Diekert, 
in the context of equations in groups. The language of solutions of equations in free groups as tuples of reduced words is EDT0L \cite{CDE2016}, as are solutions in appropriate normal forms in virtually free groups \cite{DE2017} and partially commutative groups \cite{DJK2016}. The result in \cite{CDE2016} implies that the pattern languages studied by Jain, Miasnikov and Stephan are EDT0L as well \cite{MR3050461}.

In Section \ref{sec:prelim} we give the necessary background on formal languages. In Section~\ref{sec:nonET0L} we examine the gap between ET0L and indexed languages,
and in Section \ref{sec:growth} we revisit an example of Grigorchuk and Mach\`i of a language of intermediate growth and show that it is EDT0L. In the remaining sections we show why Lindenmeyer languages are relevant for group theory by presenting several instances where L systems appear naturally. In  Section \ref{sec:primitives} we describe the set of primitives in the free group on two generators as EDT0L,  improving on the context-sensitive characterisation by Silva and Weil \cite{MR2665777}. Furthermore, we show that the set of primitives, and primitive sets, in free groups of rank higher than 2 is context-sensitive. In Section \ref{sec:Grigorchuk} we construct an explicit grammar to prove that the co-word problem in Grigorchuks group is ET0L,
and thus take a different approach to that of Holt and R\"over \cite{MR2274726}, who used nested-stack automata to describe the same set. In Section \ref{sec:BG} we show that the Bridson-Gilman normal forms for the fundamental groups of 3-manifolds or orbifolds, proved in \cite{MR1420509} to be indexed, are in fact ET0L. We conclude the paper with a list of open problems.

\section{ ET0L and EDT0L languages}\label{sec:prelim}

L systems were introduced by Lindenmayer in order to model the growth of various organisms and capture the fact that growth happens in \emph{parallel} everywhere in the organism. Therefore the rewriting system had to incorporate parallelism, as opposed to the sequential behavior of context-free grammars. The difference between sequential and parallel grammars is well illustrated by the following example (page 2 in \cite{RozS86}). 

\begin{example}
Suppose we have an alphabet $A=\{a\}$ and a single rewriting rule $a \longrightarrow a^2$, which is to be applied to $a^3$. If we apply this rule to one $a$ inside $a^3$ at a time, we get the set $\{a^i \mid i \geq 3\}$. If we apply the rule simultaneously to each $a$ in $a^3$ we obtain $a^6$ after one rewrite, and the set of words obtained via parallel rewriting is $\{a^{3 \cdot 2^i} \mid i \geq 0\}.$ 
\end{example}

There is a vast literature on Lindenmayer systems, see for example \cite{RozS86,MR1469992}, with various  acronyms such as D0L, DT0L, ET0L, HDT0L and so forth. 
The following inclusions hold: EDT0L $\subsetneq$ ET0L $\subsetneq$ indexed, and context-free $\subsetneq$ ET0L.
Furthermore, the classes of EDT0L and 
context-free languages are incomparable.

Let $V$ be a finite alphabet. A \emph{table} for $V$ is a finite subset of $V\times V^*$, that can be represented as in Figure~\ref{fig:table}.
\begin{figure}[H]
\centering
\begin{tabular}{ c c c}
 $V $ & &$V^*$ \\
 \hline
 $a$ & $\longrightarrow$ & $a, b, aba$ \\ 
 $b$ & $\longrightarrow$ & $bbabb$ \\  
 $\#$ & $\longrightarrow$ & $\#, b$ \\  
 \\
  
\end{tabular}
\caption{A table for $V=\{a,b,\# \}$.
}\label{fig:table}\end{figure}
If $(c,v)$ is in some table $t$, we say that $(c,v)$ is a \emph{rule} for $c$ and use the convention that if for some $c\in V$ no rule for $c$ is specified in $t$, then $t$ contains the rule $(c,c)$. We express the rewriting corresponding to the rule $(c,v)$ in $t$ as $c\longrightarrow^t v$. 

\begin{defn}[ET0L]
An \emph{ET0L-system} is a tuple $H=(V,A,T,I)$, where 
\begin{enumerate}
\item $V$ is a finite alphabet,
\item $A\subseteq V$ is the subset of \emph{terminal symbols},
\item $T$ is a finite set of \emph{tables} for $V$, that is,
each $t \in T$ is a finite subset 
of $V\times V^*$, and
\item $I\subseteq V^*$ is a finite set of words called \emph{axioms}.
\end{enumerate}

Let $t\in T$. We will write $u \longrightarrow^t v$ to denote that a word $v \in V^*$ can be produced from $u \in V^*$ using the rules in $t$; that is, if $u=c_1\cdots c_m$ and $v=v_1\cdots v_m$, where $c_i \in V$ and $v_i \in V^*$, we write $u\longrightarrow^t v$ to mean $v$ was obtained 
via rules $c_j \longrightarrow^t v_j$ from table $t$, applied to each $c_j$ appearing in $u$.  
More generally, $u \longrightarrow v$ signifies $u\longrightarrow^t v$ for some $t \in T$.
If there exist $u_0,\ldots,u_k \in V^*$ with $u_i\longrightarrow u_{i+1}$ for $0\leq i\leq k-1$,
then we write $u_0 \longrightarrow^* u_k$.
The language \emph{generated by $H$} is defined as
\begin{align*} L(H) & =\{ v\in A^* \mid w\longrightarrow^* v~\text{for some $w\in I$} \}. \end{align*}
A language is ET0L if it is equal to $L(H)$ for some ET0L system $H$.
 \end{defn}

\begin{defn}[EDT0L] \label{EDT0Ldef}
An \emph{EDT0L-system}
 is an ET0L system where in each table there is
 exactly one rule for each letter in $V$. 
 A language is EDT0L if it is equal to $L(H)$ for some EDT0L system $H$.
 \end{defn}

ET0L languages form a full \emph{AFL} (abstract family of languages), that is, they are closed under homomorphisms, inverse homomorphisms, intersection with regular languages, union, concatenation and Kleene closure, while EDT0L are closed under all  of the above except inverse homomorphism so do not form a full AFL \cite{Culik, Asveld}.
  
If for some words $u_1, u_2, u_3$ and tables $t_1, t_2$ we have $u_1 \longrightarrow^{t_1} u_2$ and $u_2 \longrightarrow^{t_2} u_3$ we will write $u_1 \longrightarrow^{t_1 t_2} u_3$ to denote the composition of the rewriting. This can be naturally extended to any finite sequence of rewrites. Given a regular expression $R$ over a finite set $T = \{t_1, \dots t_n\}$, where $t_1, \dots, t_n$ are tables, we will sometimes abuse the notation and use $u_1 \longrightarrow^R u_2$ to denote that there is a word $r$ in the language generated by the regular expression $R$ such that $u_1 \longrightarrow^r u_2$. Furthermore, if the system is deterministic, then every table is in fact a homomorphism on the free monoid $V^*$, and using this more algebraic notation we can give Asveld's equivalent definition for E(D)T0L languages as follows \cite{Asveld}.
\begin{defn}\label{def:et0lasfeld}\label{def:edt0lasfeld}
	Let $A$ be an alphabet and $L\subseteq A^*$. We say that $L$ is an {\em ET0L} language if there is an alphabet $C$ with $A\subseteq C$, a set $H$ of tables (i.e. finite subsets of $C \times C^*$), a regular language $\mathcal{R} \subseteq H^*$ and a letter $c\in C$ such that
\begin{displaymath}
	L = \{ w \in A^* \mid c \longrightarrow^r w \mbox{ for some } r \in \mathcal{R}\}.
\end{displaymath}
In the case when every table $h \in H$ is deterministic, i.e. each $h \in H$ is in fact a homomorphism, we write $r(c) = w$ and say that $L$ is EDT0L.

The set $\mathcal{R}$ is called the {\em regular (or rational) control}, the symbol $c$ the \emph{start symbol} and $C$ the {\em extended alphabet}.
\end{defn}

\noindent\textbf{Convention}. In any description of rational control in this paper, the maps are always applied left to right, but in algebraic settings where $f$ and $g$ are morphisms $fg(a):=f(g(a))$. 

\section{Non-ET0L languages}\label{sec:nonET0L} 
For completeness we present in this section a short survey of examples of languages which are not ET0L. The first examples turn out not to be indexed either. 

Let $\varphi \colon \mathbb{N}^+ \to \mathbb{N}^+$ be such that $\lim_{n \to \infty} \varphi(n) = \infty$ and let $U$ be an arbitrary infinite subset of $\mathbb{N}$. Set $K(\varphi, U) = \{(ba^{\varphi(k)})^{k} \mid k \in U\}$. This construction gives us an infinite family of languages. It was proved in \cite[Theorem 2]{notET0L} that $K(\varphi, U)$ is not ET0L regardless of the choice of $\varphi$ and $U$. We show in Lemma \ref{not_indexed} that these languages are not indexed.

\begin{lem}\cite[Theorem A]{shrinking} (Shrinking lemma)\label{lem:shrink}
	Let $L$ be an indexed language over a finite alphabet $\Sigma$ and let $m > 0$ be a given integer. There is a constant $k > 0$ such that each word $w \in L$ with $|w| \geq k$ can be factorised as a product $w = w_1 \dots w_r$ such that the following conditions hold:
	\begin{itemize}
		\item[(i)] $m < r \leq k$,
		\item[(ii)] $w_i \neq \epsilon$ for every $i \in \{1, \dots, r\}$,
		\item[(iii)] each choice of $m$ factors is included in a proper subproduct which lies in $L$.
	\end{itemize}
\end{lem}

\begin{lem}
	\label{not_indexed}
	The language $K(\varphi, U)$ is not indexed for any choice of $\varphi$ and $U$.
\end{lem}
\begin{proof} The proof follows that of \cite[Corollary 4]{shrinking}.
	Assume that $K(\varphi, U)$ is indexed and fix $m = 1$. Let $k \in \mathbb{N}$ be given by  Lemma~\ref{lem:shrink}.
	
	Pick $k' \in U$ such that $k' > k$ and $k' = \min( \varphi^{-1}(\varphi(k')))$. Note that such $k'$ exists as $\lim_{n \to \infty}\varphi(n) = +\infty$. Then $w = (ba^{\varphi(k')})^{k'}$ can be factorised as $w = w_1 \dots w_r$, where $1 < r \leq k < k'$. As $r < |w|_b = k'$, by the pigeonhole principle we see that there is some $1 \leq i \leq r$ such that $|w_i|_b \geq 2$, i.e. $w_i$ contains $ba^{\varphi(k')}b$ as a subword. By Lemma~\ref{lem:shrink}, $w_i$ can be included in a proper subproduct $v$ which lies in $K(\varphi, U)$. However, $v$ contains $ba^{\varphi(k')} b$, so $v = (ba^{\varphi(k'')})^{k''}$ for some $k'' \in \mathbb{N}$ such that $\varphi(k') = \varphi(k'')$. As $v$ is a proper subproduct of $w$ we see that $k'' < k'$. However, this is a contradiction  with $k' = \min( \varphi^{-1}(\varphi(k')))$, thus  $K(\varphi, U)$ is not indexed. 
\end{proof}
Note that when $\varphi$ is the identity function and $U = \mathbb{N}$ this was already established in \cite[Theorem 5.3]{MR0416123} and later in \cite[Corollary 4]{shrinking}.

An explicit example of an indexed language which is not ET0L was given in \cite{MR554283} as the language of tree-cuts. Informally speaking, a tree-cut is a sequence of binary strings encoding the list of leaves of a proper rooted binary tree, i.e. a tree in which every vertex has exactly zero or two children. Formally, tree cuts can be defined in a recursive manner:
	\begin{enumerate}
		\item[(i)] the sequence containing only the empty string, i.e. $(\epsilon)$, is a tree-cut;
		\item[(ii)] suppose that $u_1, \dots, u_k, v_1, \dots, v_l \in \{0,1\}^*$ such that the sequences $(u_1, \dots, u_k)$ and $(v_1, \dots, v_l)$ are tree-cuts, then the sequence $$(0u_1, \dots, 0u_k,1v_1, \dots, 1v_l)$$ is a tree-cut.
	\end{enumerate}
	A sequence $(v_1, \dots, v_k)$, where $v_1, \dots, v_k \in \{0,1\}$ is a tree-cut if and only if it can be obtained by repeatedly applying to the above mentioned rules. For example, the sequence $(000, 001, 01, 10, 11)$ is a tree-cut, but the sequence $(000,111)$ is not. For more detail on tree-cuts see for example \cite[Section 3]{MR554283}.
	
Let $a,b$ denote symbols distinct from $0$ and $1$. The {\em language of cuts} is then defined as
\begin{displaymath}
	L_0 = \left\{a v_1 0 bv_1 1 \dots a v_k 0 b v_k 1 \mid (v_1, \dots, v_k) \mbox{ is a cut} \right\} \subseteq \{0,1,a,b\}^*.
\end{displaymath}
Carefully checking the proof of \cite[Lemma 3.3]{MR554283} one can verify that the language $L_0$ is accepted by a nested stack automaton and hence $L_0$ is indexed. It is then proved in \cite[Lemma 3.4]{MR554283} that $L_0$ is not ET0L.

Another example of an indexed language that is not ET0L is given in \cite[Corollary 2]{ehrenfeucht1976relationship}. In fact, the paper suggests an infinite family of such languages. However, the statement of \cite[Theorem 3]{ehrenfeucht1976relationship} contains a misprint, so we give the correct statement here. For a word $w \in \Sigma^*$ let $\overline{w}$ denote its mirror image, i.e. if $w = x_1 \dots x_n$, where $x_1, \dots, x_n \in \Sigma$, then $\overline{w} = x_n \dots x_1$.
\begin{thm}(see \cite[Theorem 3]{ehrenfeucht1976relationship})
	Let $\Sigma$ be a finite alphabet and let $\Sigma'$ be a copy of $\Sigma$ distinct from $\Sigma$. Let $h \colon \Sigma^* \to \Sigma'^*$ be a homomorphism defined by $h(x) = x'$ for every $x \in \Sigma$.
	
	Let $K$ be a context-free language over $\Sigma$ such that $K$ is not EDT0L. Then the language \begin{displaymath}
	M_K = \{k \overline{h(k)} \mid k \in K\} \subseteq (\Sigma \cup \Sigma')^*
\end{displaymath}
is indexed but not ET0L.
\end{thm}
It is a well known fact (see \cite[Theorem 9]{ehrenfeucht1974some}) that if $K$ is the Dyck language on at least 8 letters, then $K$ is context-free but not EDT0L, hence $M_K$ is indexed but not ET0L.
In fact we observe below in Proposition~\ref{prop:free2notEdt0l} that the word problem for free groups of rank at least 2 is not EDT0L.

\section{Growth of languages} \label{sec:growth}

The growth of a language $L\subseteq \Sigma^*$ is the function $f:\N\ra\N$ such that $f(n)$ is the number of words in $L$ of length $n$.
If $\Sigma$ is finite then the growth function is at most exponential. A language has {\em intermediate growth} if  for any $\alpha, \beta > 1$ there is an integer $N \in \mathbb{N}$ such that for all $n > N$ we have $n^\alpha < f(n) < \beta^n$.
Bridson and Gilman showed that there are no context-free languages of intermediate growth \cite{BGilmanCF}, whereas Grigorchuk and Mach\`i  \cite[Theorem 1]{MR1678812} give the following example of a language of intermediate growth which is recognisable by a one-way deterministic non-erasing stack automaton (1DNESE), so is indexed.
They define the set $$A:=\{ab^{i_1}ab^{i_2} \dots ab^{i_k} \mid 0\leq i_1 \leq \dots \leq i_k, k \in \N\}$$ over the alphabet $\{a,b\}$.

\begin{prop}\label{prop:MG}
The language $A$ is EDT0L. \end{prop}
\begin{proof}
We show that $A$ is generated by an EDT0L system as in Definition \ref{def:edt0lasfeld}, with the following data: the extended alphabet is $\{a,b,q,q'\}$, the start symbol $q$, and the maps are $h_a$, $h_b$, $h_{\$}$ with $h_a(q)=qaq'$, $h_b(q)=qb$, $h_b(q')=q'b$, $h_{\$}(q)=h_{\$}(q')=\epsilon$. 

The rational control is given by $\mathcal{R}=\{h_a,h_b\}^*h_ah_{\$}$.

Let $g \in \{h_a,h_b\}^*$. We can easily prove by induction on the length of $g$ as a word over $\{h_a,h_b\}$ that $g(q)=qb^{i_1}aq'b^{i_2}a \dots aq'b^{i_k}$ where $0\leq i_1 \leq \dots \leq i_k$, i.e the word starts with $q$, and then only $q'$ appears. Then applying $h_a h_{\$}$ to $g(q)$ produces a word which starts with $a$ and contains no more $q, q'$.
\end{proof}
  \begin{cor}There exist EDT0L (and ET0L) languages of intermediate growth.
  \end{cor}

\section{Primitives and bases in free groups}\label{sec:primitives}
Given an alphabet $\Sigma$, we will always assume that a tuple $(w_1, \dots, w_t)$, where $w_1, \dots,  w_n \in \Sigma^*$, is encoded as a string $w_1\#\dots\#w_2$, where the symbol $\#$ is distinct from the ones contained in $\Sigma$. 

A {\em free basis} of a free group is a tuple of elements that freely generate the group, and a {\em primitive (element)} is an element that belongs to some free basis of the group. In this section we show that the set of bases and the set of primitives, written as reduced words over the standard free basis, is EDT0L for free groups of rank $2$, and context-sensitive for higher rank. 
 It is an open question whether the  higher rank  context-sensitive characterisation can be lowered to indexed or E(D)T0L.

\subsection{Bases and primitives in the free group $F_2$ are EDT0L}

In \cite{MR2665777} Silva and Weil showed that the set of primitives in the free group  $F_2$ on generators $\{a,b\}$, written as reduced words over $\{a, b, a^{-1}, b^{-1}\}$, is context-sensitive and not context-free. In this section we show that both the set of bases and the set of primitives are EDT0L. 

\begin{prop}\label{prop:primitives}
The set of free bases and the set of primitives in $F_2$, as reduced words, are EDT0L. 
\end{prop}

\begin{proof}
It is a classical result due to Nielsen that two elements $g,h \in F_2$ form a basis of $F_2$ if and only if the commutator $[g,h]$ is conjugate either to $[a, b]$ or $[b, a]$. Hence the set of bases in $F_2$ can be seen as the values $X, Y$ satisfying the equations $$Z[X,Y]Z^{-1}=[a,b]$$ or $$Z[X,Y]Z^{-1}=[b,a]$$ where $Z$ can take any value.

By Corollary 2.2 in \cite{CDE2016}, for any equation in a free group the set of solutions, or any projections thereof, written as tuples in reduced words, is EDT0L. In this case the values of $(X,Y)$ will provide the set of bases, and the values of $X$ the set of primitives.
\end{proof}

\subsection{Primitives in free groups of higher rank}

While in $F_2$ the set of free bases can be expressed in terms of solutions to equations, this does not hold in higher rank. It was shown in \cite{definable} that the set of primitives in free groups of higher rank is not   definable in the first order theory of the group, and thus we cannot use our previous approach to give an EDT0L characterisation in this case.
Another approach is to produce the set of primitives by applying all the maps in $\Aut(F_k)$ to a free basis element. Since $\Aut(F_k)$ is finitely generated, we would apply all the maps in $\Aut(F_k)$, written as words over the generators of $\Aut(F_k)$ - which can be seen as rational control, to a basis element which would play the role of the axiom. The set thus obtained would be EDT0L and contain all the primitives, but neither as reduced nor unique words. 

Since we are interested in establishing a formal language characterisation for primitives as reduced words, we use a different approach, based on Stallings' folding of labeled graphs, and an algorithm for testing primitivity in a free group given in \cite{CliffordGoldstein}, to show here that they are context-sensitive. For the sake of keeping this note succinct, we do not define here the {\em Stallings graph} or what is meant by folding, but refer the reader to references such as \cite{stallings} and \cite{KapMya_folding}. If $\Gamma$ is the Stallings graph of a finitely generated subgroup of a free group $F$, a \emph{pinch} will signify the identification of two distinct vertices of $\Gamma$. 

A \emph{primitive set} (as opposed to the set of primitives) in a free group $F$ is a set of elements that can be extended to a basis of $F$.

\begin{lem} \cite[Theorem 4.4]{CliffordGoldstein}
	\label{lemma:fold_and_pinch}
	Let $n\leq k$ and let $W = \{w_1, \dots, w_n\}$ be a set of reduced words in $F_k = F(X)$, where $X = \{x_1, \dots, x_k\}$, such that for every $x \in X$ either $x$ or $x^{-1}$ appears in some $w \in W$. Let $\Gamma$ be a graph representing the subgroup of $F_k$ generated by $W$. Then $W$ is a primitive set if and only if there is a sequence of pinches and folds, containing exactly $k-n$ pinches, that transforms $\Gamma$ into the elementary wedge on $X$.
\end{lem}
In the case that not every generator or inverse appears in the list of $w_i$ then the Lemma \ref{lemma:fold_and_pinch} can still decide primitivity by taking a smaller value for $k$.

A standard way to represent a Stallings graph $\Gamma=(V\Gamma, E\Gamma)$ over an alphabet $X$ would be as a collection of labeled oriented edges, where an edge $\gamma \in E\Gamma$ is given as a triple $(i(\gamma), t(\gamma), l(\gamma))$, with $i(\gamma) \in V\Gamma$ the initial vertex of $\gamma$, $t(\gamma) \in V\Gamma$ the terminal vertex and $l(\gamma) \in X \cup X^{-1}$ the label of $\gamma$. Obviously, a word $w \in F(X)$ can be represented by a graph $\Gamma_w$ such that $|V\Gamma_w|, |E\Gamma_w| \leq n$, where $n = |w|$. However, as there might be up to $n$ vertices, and one needs $\log(n)$ space to encode the names of the vertices, the na\"ive implementation of $\Gamma$ would require $n \log(n)$ space.

We propose a different representation of the Stallings graph, along with two additional types of moves, and show that Lemma \ref{lemma:fold_and_pinch} can be realised by a linear space algorithm.
Instead of labelling an edge by a single symbol, we will label edges by subwords, that is, an edge $\gamma \in E\Gamma$ will be represented by a triple $(i(\gamma), t(\gamma), s(\gamma))$, where $i(\gamma)$ and $t(\gamma)$ are as before and $s(\gamma) \in F(X)$ is a reduced word. In this case we say that $\Gamma$ is a \emph{segment graph}. Adapting the terminology of Stallings graphs, we say that a segment graph $\Gamma$ is \emph{folded} if it does not contain a pair of incident edges whose labels share a common prefix or suffix. If a vertex $v$ is incident to two edges with a common prefix or suffix, then we say that $v$ is \emph{unfolded}, and otherwise call it {\em folded}.

Folding at a vertex $v$ applies to segment graphs in three possible ways, depending on whether two segments incident to $v$ have exactly the same label, or only common proper prefixes. More precisely, let $\gamma_1, \gamma_2 \in E\Gamma$ satisfy $i(\gamma_1) = v = i(\gamma_2)$, $s(\gamma_1) = u w_1$ and $s(\gamma_2) = u w_2$ for some $u, w_1, w_2 \in F(X)$. 
	\begin{itemize}
	\item[I.] If $w_1=w_2=\epsilon$, then this is a usual folding of graphs.
	\item[II.] If $w_1=\epsilon$ and $w_2 \neq \epsilon$, then we split $\gamma_2$ into $\gamma'_2$ and $\gamma''_2$ by introducing a new vertex $v'$ on $\gamma_2$ such that $i(\gamma'_2)=v$, $t(\gamma'_2)=v'$ and $s(\gamma'_2)=u$, and fold $\gamma_1$ with $\gamma'_2$ as in case $I$.
	\item[III.] If $w_1, w_2 \neq \epsilon$, then we introduce a new vertex $v'$ and a new edge $\gamma$, such that $i(\gamma) = v$, $t(\gamma) = v'$, $s(\gamma) = u$, and replace edges $\gamma_1, \gamma_2$ by new edges $\gamma_1', \gamma_2'$ such that $i(\gamma_1') = v' = i(\gamma_2')$, $t(\gamma_1') = t(\gamma_1)$, $t(\gamma_2') = t(\gamma_2)$, $s(\gamma_1') = w_1$ and $s(\gamma_2') = w_2$. 
	\end{itemize}

	Pinching a segment graph is defined as follows. Let $\gamma, \gamma' \in E\Gamma$ satisfy $s(\gamma) = w_1 w_2$ and $s(\gamma') = w_1' w_2'$ for some $w_1, w_2, w_1', w_2' \in F(X)$. We remove the edges $\gamma, \gamma'$, introduce new vertex $v$ and up to four new edges $\gamma_1, \gamma_2, \gamma_1', \gamma_2'$ given by
	\begin{itemize}
		\item[$\gamma_1$:] $i(\gamma_1) = i(\gamma)$, $t(\gamma_1) = v$, $s(\gamma_1) = w_1$;
		\item[$\gamma_2$:] $i(\gamma_2) = v$, $t(\gamma_2) = t(\gamma)$, $s(\gamma_2) = w_2$;
		\item[$\gamma_1'$:] $i(\gamma_1') = i(\gamma')$, $t(\gamma_1') = v$, $s(\gamma_1') = w_1'$;
		\item[$\gamma_2'$:] $i(\gamma_2') = v$, $t(\gamma_2') = t(\gamma')$, $s(\gamma_2) = w_2'$.
	\end{itemize}
	If one of $w_1, w_2, w_1', w_2'$ is empty, i.e. when we are pinching a segment with a vertex, we don't need to introduce the vertex $v$.
	
	Both foldings of type III and pinches can cause the numbers of edges and vertices to increase; the resulting graph could potentially not fit on a linear size tape, but the discussion below shows that this is not the case.

Recall that if $\Gamma$ is a Stallings graph corresponding to some finite set $W$ then the {\em core} of $\Gamma$ (or the \emph{core graph} for $\langle W \rangle$), is the smallest subgraph of $\Gamma$ in which we can read all the elements of $\langle W \rangle$ as reduced words along loops based at the base vertex of $\Gamma$. In particular, if $\Gamma$ is finite then the core of $\Gamma$ does not contain any vertices of degree one, with the exception of the base vertex of $\Gamma$. The core of a segment graph is defined in an analogous manner.

\begin{rmk}
\label{remark:base_vertex_is_irrelevant}		
(1) A set $W = \{w_1, \dots, w_n\} \subseteq F_k$ is primitive if and only if the set $W^g = \{gw_1g^{-1}, \dots, g w_n g^{-1}\}$ is primitive for some (and thus for all) $g \in F_k$. As choosing a different base vertex in the core graph of $\langle W \rangle$ is the same as taking a conjugate of $\langle W \rangle$, we will assume that the base vertex in any segment graph has degree at least $2$. 
In fact, unless the graph is homeomorphic to a loop, we may assume that the base vertex is of degree at least three.

(2) Furthermore, as the core graph will never contain any leaves other than the base vertex and as was already mentioned, we can always assume that the base vertex is of degree at least two, hence without loss of generality we may assume that the segment graph does not contain any leaves; this can be achieved by \textit{pruning}, i.e. repeatedly removing vertices of degree 1 (along with corresponding edges). Note that for finite graphs this procedure always terminates.
\end{rmk}

	We say that a segment graph $\Gamma$ is \emph{topological} if all its vertices have degree $\geq 3$, unless it consists of a single vertex, or a loop attached at a degree two vertex. Following Remark \ref{remark:base_vertex_is_irrelevant}, we see that every segment graph can be transformed into a topological one (representing a subgroup conjugate to the original) by picking a different base vertex, pruning and \textit{merging} edges: suppose $\gamma_1, \gamma_2 \in E\Gamma$ such that $t(\gamma_1) = v = i(\gamma_2)$ and $\deg(v) = 2$; then we can replace the edges $\gamma_1, \gamma_2$ by a new edge $\gamma$ such that $i(\gamma) = i(\gamma_1)$, $t(\gamma) = t(\gamma_2)$ and $s(\gamma) = s(\gamma_1) s(\gamma_2)$, and remove vertex $v$.
		
The following lemma implies that pruning will need to be done at most once in the entire process corresponding to Lemma \ref{lemma:fold_and_pinch}.

\begin{lem}
	\label{lemma:one_pruning}
	Let $\Gamma$ be a folded graph that does not contain any vertex of degree $1$ and let $\Gamma'$ be a graph obtained from $\Gamma$ by performing a pinch and a folding. Then $\Gamma'$ does not contain vertices of degree $1$.
\end{lem}
\begin{proof}
	Suppose that there is $v' \in V\Gamma'$ such that $\deg(v') = 1$. Without loss of generality we may assume that $v'$ is the initial vertex of its unique adjacent edge $\gamma \in E\Gamma'$. Note that the pinch and the subsequent folding uniquely define a surjective morphism of graphs $f \colon V\Gamma_0 \to V\Gamma'$, where $\Gamma$ can be obtained from $\Gamma_0$ by merging. Note that $\Gamma_0$ is folded and it does not contain vertices of degree $1$. Let $v \in f^{-1}(v')$ be arbitrary. Again, without loss of generality we may assume that $v$ is the initial vertex of all of its adjacent edges. It follows that all the edges adjacent to $v$ must share a common prefix with $\gamma$, meaning that $v$ is not folded folded unless $\deg(v) = 1$.  As $\Gamma_0$ is folded, $\deg(v) = 1$ for every $v \in f^{-1}(v')$, which contradicts the fact that $\deg(u) > 1$ for every $u \in V\Gamma_0$. Thus $\Gamma'$ does not contain vertices of degree $1$.
\end{proof}
	
The {\em topological rank} of a graph is the rank of its fundamental group.
		
\begin{lem}\label{lemma:Euler}
Let $\Gamma$ be a finite topological graph of topological rank $r$ with $r > 1$. Then $|E \Gamma|\leq 3r-3$ and $|V \Gamma|\leq 2r-2$.
\end{lem}
\begin{proof}
As $r > 1$ we see that $\deg(v) \geq 3$ for every $v \in V\Gamma$. From the Euler characteristic formula $r=|E\Gamma|-|V\Gamma|+1$, and since in any graph $2|E\Gamma|=\sum_{v\in V\Gamma}\deg(v)$ by assumption $2|E\Gamma|\geq 3|V\Gamma|$, which gives the inequalities above.
\end{proof}

\begin{rmk}
\label{remark:graph_observations}
	Folding does not cause the topological rank of a graph to increase, but pinching can increase it by 1.
\end{rmk}

\begin{cor}
	\label{lemma:folding_estimate}
	Let $\Gamma$ be a finite topological segment graph of topological  rank $r$ with a single unfolded vertex, and let $\Gamma'$ be the folded topological graph obtained from $\Gamma$ by folding, merging and pruning. 
	
	If $r = 1$ then $|V\Gamma'| = |E\Gamma'| = 1$, otherwise $|V\Gamma'|\leq 2r - 2$ and $|E\Gamma'|\leq 3r - 3$.
\end{cor}
\begin{proof}
Let $r'$ denote the topological rank of the graph $\Gamma'$; note that $r' \leq r$ by Remark \ref{remark:graph_observations}.

If $r' = 1$ then, following the definition of topological graphs, we see that $\Gamma'$ is a single vertex with an attached loop and $|V\Gamma'| = |E\Gamma'| = 1$.

If $r' > 1$ then the result follows by Lemma \ref{lemma:Euler}.
\end{proof}

\begin{prop}
	Let $F_k$ be a free group of rank $k\geq 2$. Then the primitive sets in $F_k$ can be recognised in linear space. 
\end{prop}

\begin{proof}
We represent a graph with $t$ edges on the tape of a Turing machine by the string
$$\#v_{i_1}| v_{i_2}| w_1\#\#\dots\#\#v_{i_{2t-1}}| v_{i_{2t}}| w_t\#,$$
where  $v_{i_j}$ are binary numbers and the factor  $\#v| v'| w\#$ represents the segment $\gamma$ with $i(\gamma)=v$, $t(\gamma)=v'$ and $s(\gamma)=w$. 

On input a list of words $w_1,w_2,\dots ,w_n$,
 if $n\geq k$ we return  \texttt{NO}, else we write on the tape 
$$\#v_0| v_0| w_1\#\#v_0| v_0| w_2\#\#\dots\#\#v_0| v_0| w_n\#$$
which describes the bouquet of loops labeled $w_i$ at a vertex $v_0$. 
Let $k'$ be the rank of the free group generated by the letters appearing in the $w_i$, so $k'\leq k$.

We perform folds of types I, II, III by modifying the tape as follows: 
\begin{itemize}
	\item[I.] Scan the tape to find  factors $\#v_i| v_j| u\#$ and $\#v_i| v_k| u\#$; erase the second factor and replace $v_k$ by $v_j$ everywhere on the tape.
	\item[II.] Scan the tape to find $\#v_i| v_j| u\#$ and $\#v_i| v_k| up\#$; overwrite the second factor with $\#v_j| v_k| p\#$.
	\item[III.] Scan the tape to find $\#v_i| v_j| uap\#$ and $\#v_i| v_k| ubq\#$ with $a\neq b, a,b \in X^{\pm 1}$; erase both factors and write $\#v_i| v| u\#\#v| v_j| ap\#\#v| v_k| bq\#$ where the binary number $v$ is some value not already in use.
\end{itemize}

For pruning we scan the tape to find a factor $\#v|v'|u\#$ such that the $v$ (or $v'$) does not appear in any other factor. If such a factor is found, we delete it. 

Similarly, for merging we scan the tape for a pair of factors $\#v|v'|u_1\#$ and $\#v'|v''|u_2\#$ (or $\#v''|v'|u_2\#$) such that $v''$ does not occur in any other factor. If such a pair is found, we replace them by $\#v|v''|u_1 u_2\#$ (or $\#v|v''|u_1 u_2^{-1}\#$, respectively).

We perform a pinch by choosing either:
\begin{enumerate}
	\item two distinct vertices $v_i,v_j$ and replacing $v_j$ by $v_i$ everywhere on the tape;
	\item a vertex $v$ and a segment $(v_i,v_j,w_1w_2)$ with $|w_i|>0$ and replacing $\#v_i| v_j| w_1w_2\#$ by $\#v_i| v| w_1\#\#v| v_j| w_2\#$;
	\item two segments $(v_i,v_j,w_1w_2), (v_p,v_q,w_3w_4)$ with $|w_i|>0$, and replacing their encodings by $\#v_i| v| w_1\#\#v| v_j| w_2\#\#v_p| v| w_3\#\#v| v_q| w_4\#$ where  $v$ is some value not already in use.
 \end{enumerate}
 Note that folding and pruning moves decrease the number of letters appearing as labels of segments, and pinching and merging moves preserve this number.

The procedure starts by performing folding moves  I, II, III in any order exhaustively, and each folding is followed by all possible pruning and merging. Then a pinch is applied, and the previous two steps repeat $k'-n$ times.

Following Lemma \ref{lemma:one_pruning} we see that we will only need to perform pruning moves before the first pinch. Note that during this process, the number of letters present in the labels might decrease. In this case we need to decrease $k'$ accordingly.

Return  \texttt{YES} if the tape contains $$\#v| v| a_{i_1}\#\#\dots \#\#v| v| a_{i_{k'-n}}\#$$ with all $i_j$ distinct and $a_{i_j}\in X^{\pm 1}$,   else return  \texttt{NO}.

Termination of the algorithm is guaranteed since each fold strictly decreases the number of letters from $X^{\pm 1}$ appearing on the tape as labels of segments. The algorithm accepts precisely the primitive sets in $F_k$ by Remark~\ref{remark:base_vertex_is_irrelevant} and Lemma~\ref{lemma:fold_and_pinch}.

Since the rank of all topological segment graphs considered is smaller than the rank of the ambient free group, by Corollary~\ref{lemma:folding_estimate} the number of vertices in use at any time is at most $2k-2$, so in binary notation each vertex requires at most $\log(2k-2)$ space,  the total number of segments is at most $3k-3$, and the number of letters from $X^{\pm 1}$ appearing on the tape is at most $\sum_{i=1}^n|w_i|=N$. Thus the amount of space required at any time in the process is at most 
$$(3k-3)(2\log(2k-2)+4)+\sum_{i=1}^n|w_i|,$$ which is linear in the input size.
\end{proof}

From the proposition it follows that for every $k$ there is a linearly bounded Turing machine $T_k$ which recognises primitive sets in $F_k$, hence we can state the following corollary.
\begin{cor}\label{prop:primitivesHighRank}
	Let $F_k = F(X)$ be a free group over $X$, where $|X| = k$ and $n\leq k$. Then 
	\begin{align*}
		\mathcal{P}_{k,n} 	&= \left\{w_1\# \dots \# w_n \mid \{w_1, \dots w_n\} \mbox{ is a primitive set in $F_k$}\right\}\\
					 	&\subseteq \left(X\cup X^{-1} \cup \{\#\}\right)^*
	\end{align*}
	is a context-sensitive language. 
	
	In particular, the set $P_k$(=$\mathcal{P}_{k,1}$) of primitive elements is context-sensitive.
\end{cor}

\section{Co-word problem for Grigorchuk group is ET0L}\label{sec:Grigorchuk}
    In this section we show that the co-word problem for the Grigorchuk group is ET0L, improving on Holt and R\"over's result in \cite{MR2274726}, where they showed it is indexed. It is still an open question  whether the co-word problem for the Grigorchuk's group is context-free.
    
    \subsection{Generators and the word problem}
    \label{sub:wp_algorithm}
    We refer the reader to \cite{delaharpe} for more details, here we give the essentials for our proof. 
    Let $\T$ denote the infinite rooted binary tree and let $\A = \Aut(\T)$. Note that $\A \simeq \A \wr C_2 \simeq (\A \times \A) \rtimes C_2$, where $C_2 = \langle \alpha \mid \alpha^2 = 1\rangle$, so every $g \in \A$ can be uniquely expressed as $g = (g_L, g_R) \alpha_g$ for some $g_L, g_R \in \A$ and $\alpha_g \in C_2$. The Grigorchuk group is $G = \langle a, b, c, d \rangle \leq \A$, where the generators $a,b,c,d \in \A$ are given by
    \begin{displaymath}
        \begin{array}{cccc}
            a = (1,1)\alpha, &b = (a,c)1, &c = (a,d)1, &d = (1,b)1.
        \end{array}
    \end{displaymath}
   One can easily verify that the following identities hold in $G$:
    \begin{equation}
        \label{eq:reductions}
        \begin{array}{c}
            a^2 = b^2 = c^2 = d^2 = 1,\\
            bc = cb = d, \quad cd = dc = b, \quad db = bd = c.
        \end{array}
    \end{equation}
    Starting with an arbitrary word $w \in \{a,b,c,d\}^*$, one can rewrite $w$ via the identities in (\ref{eq:reductions}) to a word $w'$ which represents the same element in $G$ and does not contain any of
    \begin{displaymath}
    	aa,bb,cc,dd, bc, cb, bd, db, cd, dc
	\end{displaymath}
	 as a subword. Thus $w'$ has the form $w' = x_0 a x_1 \dots x_{n-1} a x_n$, where $x_i \in \{1,b,c,d\}$ for $0 \leq i \leq n$ and $x_i \neq 1$ for $1 < i < n-1$, so we say that $w'$ is \emph{alternating} or \emph{reduced}. Every non-trivial element $g \in G$ can be represented by a reduced word. The following remark has an analogous proof to that of the uniqueness of reduced words in free groups (see \cite[Section I.1]{MR0577064}).
    \begin{rmk}
    	\label{remark:reduction_uniqueness}
    	Every word $w \in \{a,b,c,d\}^*$ can be rewritten via the identities in (\ref{eq:reductions}) to a unique reduced word.
    \end{rmk}
   	However, not every reduced word represents a nontrivial element: the word $dadadada$ is reduced, yet it represents the identity in $G$.
    
    Let $G_1$ be the subgroup of $G$ consisting of all elements that fix the first level of $\T$. Then $|G:G_1| = 2$, $G_1 = \langle b, c, d, aba, aca, ada \rangle$ and     \begin{displaymath}
        \begin{array}{ccc}
            aba = (c,a)1, &aca = (d,a)1, &ada = (d,1)1.
        \end{array}
    \end{displaymath}
    Every element $g \in G_1$ can be expressed as an alternating sequence $g = x_0 x^a_1 \dots x_{n-1} x^a_n$, where $x_i \in \{1,b,c,d\}$ and $x^a_i \in \{1, aba, aca, ada\}$ for $0 \leq i \leq n$ such that $x_i \neq 1$ if $i>0$ and $x^a_i \neq 1$ if $i < n$. 
    
    Let $\phi_L, \phi_R \colon G_1 \to G$ be the group homomorphisms defined as
    \begin{displaymath}
        \phi_L \colon \begin{cases}
            b   &\rightarrow a,\\
            c   &\rightarrow a,\\
            d   &\rightarrow 1,\\
            aba &\rightarrow c,\\
            aca &\rightarrow d,\\
            ada &\rightarrow b,
        \end{cases}
        \quad
        \phi_R \colon \begin{cases}
            b   &\rightarrow c,\\
            c   &\rightarrow d,\\
            d   &\rightarrow b,\\
            aba &\rightarrow a,\\
            aca &\rightarrow a,\\
            ada &\rightarrow 1,
        \end{cases}
    \end{displaymath}
    and define $\phi \colon G_1 \to G \times G$ as $\phi(g) = (\phi_L(g), \phi_R(g))$. Then $\phi$ is injective, and $|\phi_i(w)| < \frac{1}{2}|w| + 1 < |w|$ for every reduced $w \in \{a,b,c,d\}^*$ representing some element in $G_1$ such that $|w| > 1$. Let $L_S$ be the set of words with an odd number of $a$'s, that is,
    \begin{displaymath}
    	L_S = \{w \in \{a,b,c,d\}^* \mid |w|_a \equiv 1 \mod 2\}.
	\end{displaymath}
	Notice that the parity of $|w|_a$ is invariant under the rewrite rules (\ref{eq:reductions}). Indeed, if $w \in L_S$ then $w \neq_G 1$. 
	
	 The following is the outline of the word problem algorithm in Grigorchuk's group (see \cite[Section VIII.E]{delaharpe} for more details and Figure \ref{fig:example} for an example). We use $\epsilon$ to denote the empty string.
	\begin{enumerate}
		\item reduce $w$
		\item if $w = \epsilon$ answer \texttt{YES} (meaning $w =_G 1$),
		\item if $w \in L_S$ answer \texttt{NO} (meaning $w \neq_G 1$),
		\item otherwise answer $(\phi_L(w) =_G 1 \texttt{ and } \phi_R(w) =_G 1)$.
	\end{enumerate}
	Obviously, $w \neq_G 1$ if and only if there are sequences $w_0, \dots, w_n \in \{a, b, c, d\}$ and $\phi_1, \dots, \phi_n \in \{\phi_L, \phi_R\}$  such that $w_n =_G w$, $w_{i-1} =_G \phi_i(w_i)$ for $i = 1, \dots, n$ and $w_0 \in L_S$. The main idea behind our grammar is to invert this process, i.e. start with a word in $w_0 \in L_S$ and generate a sequence of words $w_1, \dots, w_n$ such that $w_{i-1} =_G \phi_i(w_i)$.
     
    \subsection{ET0L grammar}
    In this subsection we introduce the ET0L grammar used to generate the co-word problem for Grigorchuk's group and give an informal explanation of the roles of the corresponding tables.
    
    Our grammar works over the extended alphabet
    \begin{displaymath}
    	\Sigma = \{S_0, S_1, a, b, c, d, \delta, \#\},
	\end{displaymath}
	where $S_0$ is the start symbol, along with tables $s, p, h_L, h_R, u, t$ and rational control
	\begin{displaymath}
		\mathcal{R} = s^* \left\{p^* \{h_L, h_R \} u^* t \right \}^*.
	\end{displaymath}
		
	The process of generating words can be split in three phases. It is important to note that the second and third phase can occur multiple times, as can be seen in Example \ref{ex:generation}.\\
	
\subsubsection*{Phase 1: Generate $L_S$} 
	
	First we generate the language $L_S$ of words with an odd number of occurrences of $a$. This is done by table \ref{tab:seeding}, which can be seen as the grammar version of a two state finite automaton producing the regular language $L_S$. We call the words in $L_S$ \emph{seeds}.    
    \begin{equation*}
    	\tag{$s$}
    	\label{tab:seeding}
    	\begin{array}{lcl}
        	\multicolumn{3}{l}{\mbox{Generate seeds}}\\
        	\hline
         	S_0 & \rightarrow   &  a S_1, b S_0, c S_0, d S_0\\
         	S_1 & \rightarrow   & a S_0, b S_1, c S_1, d S_1, \epsilon\\
          \end{array}
    \end{equation*}

	\begin{lem}
		\label{lemma:seeding}
		Let $w \in \{a,b,c,d\}^*$. Then $w \in L_S$ if and only if $S_0 \longrightarrow^{s^*} w$. In particular, if $S_0 \longrightarrow^{s^*} w$ then $w \neq_G 1$.
	\end{lem}
	
\subsubsection*{Phase 2: Invert $\phi_L$ and $\phi_R$}

    Once we have produced a seed $w_0$, we generate a sequence of words $w_1, \dots, w_n$ such that $w_{i-1} =_G \phi_i(w_i)$, and we create $w_i$ from $w_{i-1}$ by `inverting' the maps $\phi_L$ and $\phi_R$; this is achieved via tables \ref{tab:pulling}:  
     \begin{displaymath}
    	\tag{$h_L, h_R$}
    	\label{tab:pulling}
        \begin{array}{lcl}
            \multicolumn{3}{l}{\mbox{Invert $\phi_L$}}\\
            \hline
            a 			&\rightarrow& b,c\\
            b 			&\rightarrow& ada\\
            c 			&\rightarrow& aba\\
            d 			&\rightarrow& aca\\
            \delta 		&\rightarrow& d
        \end{array}
        \quad
        \begin{array}{lcl}
            \multicolumn{3}{l}{\mbox{Invert $\phi_R$}}\\
            \hline
             a 			&\rightarrow& aba, aca\\
            b 			&\rightarrow& d\\
            c 			&\rightarrow& b\\
            d 			&\rightarrow& c\\
            \delta 		&\rightarrow& ada
        \end{array}
    \end{displaymath}

    Table \ref{tab:preprocessing} introduces a new symbol $\delta$, which serves as a placeholder for the empty word that resulted from applying $\phi_L$ and $\phi_R$ (recall that $\phi_L(d) = 1$, $\phi_R(ada) = 1$).
     \begin{displaymath}
	\tag{$p$}
	\label{tab:preprocessing}
	\begin{array}{lcl}
		\multicolumn{3}{l}{\mbox{Insert $\delta$}}\\
        		\hline
         	a &\rightarrow& a, \delta a, a \delta \\
		b &\rightarrow& b, \delta b, b \delta \\
		c &\rightarrow& c, \delta c, c \delta \\
		d &\rightarrow& d, \delta d, d \delta
		\end{array}
	\end{displaymath}

    \begin{lem}
		\label{lemma:pulling}
		Let $w, w' \in \{a,b,c,d\}^*$ and let $i \in \{L,R\}$ be given. Then $w=_G \phi_i(w')$ if and only if $w \longrightarrow^{p^* h_i} w'$. In particular, if $w \longrightarrow^{p^* h_i} w'$ and $w \neq_G 1$ then $w' \neq_G 1$.
	\end{lem}
The proof follows immediately from an analysis of the tables.

\subsubsection*{Phase 3: Use group relations and insert trivial subwords}
     
    The following table inverts the reduction process induced by the identities in (\ref{eq:reductions}). We introduce the symbol $\#$ as a placeholder for the reduction process. This means that $\#$ signifies the position of a subword that reduces to the trivial word.
    \begin{displaymath}
    	\tag{$u$}
    	\label{tab:unreducing}
        \begin{array}{lcl}
            \multicolumn{3}{l}{\mbox{Unreduce}}\\
            \hline
            a	&\rightarrow& a, \#a, a\#\\
            b	&\rightarrow& b, \#b, b\#, c \# d, d \# c\\
            c	&\rightarrow& c, \#c, c\#,  b \# d, d \# b\\
            d   &\rightarrow& d, \#d, d\#, b \# c, c \# b\\            
            \#  &\rightarrow&  \#, a \#a, b \# b, c \# c, d \# d           
        \end{array}
    \end{displaymath}

    At this stage we remove all the occurrences of the placeholder $\#$. This is achieved by table \ref{tab:tidying}:
    \begin{displaymath}
		\tag{$t$}
		\label{tab:tidying}
			\begin{array}{lcl}
				\multicolumn{3}{l}{\mbox{Tidy \#}}\\
        				\hline
         			 \#	&\rightarrow& \epsilon
		\end{array}
	\end{displaymath}
     
           	\begin{lem}
		\label{lemma:reduction}
		Let $w, w' \in \{a,b,c,d\}^*$ be arbitrary. Then $w$ can be obtained from $w'$ by reducing rules (induced by the identities in (\ref{eq:reductions})) if and only if $w \longrightarrow^{u^*t} w'$. In particular, if $w \longrightarrow^{u^*t} w'$ and $w \neq_G 1$ then $w' \neq_G 1$.
	\end{lem}

        \begin{example}
        \label{ex:generation} Figure~\ref{fig:example} demonstrates the word problem algorithm on input $bcddacbabcaa$, showing it to be nontrivial.
       
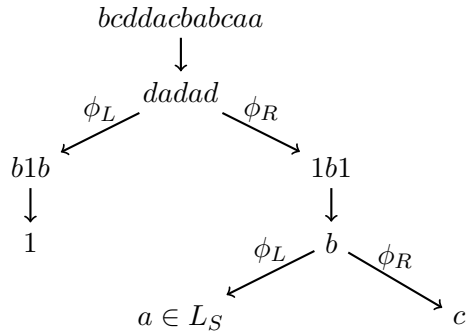
\begin{figure}[h!]
\begin{center}
\begin{tikzpicture}[thick,scale=1, every node/.style={scale=1}]

\draw (3,5) node (A) {$bcddacbabcaa$};
\draw (3,4) node (B) {$dadad$};
\draw (1,3) node (C) {$b1b$};
\draw (5,3) node (D) {$1b1$};
\draw (5,2) node (E) {$b$};
\draw (3,1) node (F) {$a \in L_S$};
\draw (6.7,1) node (G) {$c$};
\draw (1,2) node (H) {$1$};

\draw[->] (A) -- (B);
\draw[->] (B) -- node[above] {$\phi_L$} (C);
\draw[->] (B) --  node[above] {$\phi_R$}  (D);
\draw[->] (D) -- (E);
\draw[->] (E) -- node[above] {$\phi_L$} (F);
\draw[->] (E) -- node[above] {$\phi_R$} (G);
\draw[->] (C) -- (H);
\end{tikzpicture}
\caption{Word problem algorithm applied to $bcddacbabcaa $.}\label{fig:example}
\end{center}
\end{figure}

 This word can be obtained from our grammar as follows.

	\[	\begin{array}{lll}	
			  \text{Phase 1:} & \ \ \ &   S_0 \longrightarrow^{s} a S_1 \longrightarrow^{s} a\\
			  \\
    \text{Phase 2:} && a \longrightarrow^{h_L} b\\
    \\
   \text{Phase 3:} & &  b \longrightarrow^t b \\
     \\
     
    \text{Phase 2:} & & b \longrightarrow^{p} \delta b \longrightarrow^{p} \delta b \delta \longrightarrow^{h_R} dadad\\
    \\
       \text{Phase 3:} & &
       		dadad 	 \longrightarrow^{u} b\#c\#ac\#bab\#c \longrightarrow^{u}  b\#cd\#dac\#bab\#c \# \\& &\longrightarrow^{u} b\#cd\#dac\#bab\#c a\#a 
				\longrightarrow^{t} bcddacbabcaa.
			\end{array}\]

       \end{example}
        
\subsection{Proof that the ET0L grammar generates exactly the co-word problem}\label{sec:GProof}

Let $L'$ be the language generated by the ET0L grammar with alphabet $\Sigma$, tables $s, p, h_L, h_R, u, c$ given above and rational control 
\begin{displaymath}
	\mathcal{R} = s^* \left\{p^* \{h_L, h_R \} u^* t \right \}^*.
\end{displaymath}
Note that the rational control $\mathcal{R}$ is equivalent to $s^* \{p^*h_L u^*t, p^*h_R u^*t\}^*$.
Following Definition \ref{def:et0lasfeld}, $L'$ is ET0L. Now consider $L = L' \cap \{a,b,c,d\}^*$. As the class of ET0L languages is closed under intersection with regular languages, $L$ is an ET0L language. Intersecting $L'$ with $\{a,b,c,d\}^*$ effectively discards all words containing symbols $S_0, S_1$. With this in mind, we will assume that we are only working with words that do not contain symbols $S_0$ and $S_1$.
				
	Combining Lemmas~\ref{lemma:seeding},~\ref{lemma:pulling} and \ref{lemma:reduction} we can immediately show that our grammar produces words that represent non-trivial elements in the group.
	\begin{lem}
		\label{lemma:consistency}
		Let $w \in \{a,b,c,d\}^*$ be arbitrary. If $w \in L$ then $w \neq_G 1$.	
	\end{lem}
	\begin{proof}
		If $w \in L$ then, by definition, there is a sequence of words $$w_0, w_1, w_1', \dots, w_n, w_n' \in \{a,b,c,d\}^*$$ and a sequence of tables $h_1, \dots, h_n \in \{h_L, h_R\}$ such that
		\begin{enumerate}
			\item $s \longrightarrow^{s^*} w_0$,
			\item $w_{i-1} \longrightarrow^{p^*h_i} w_{i}'$ for $i = 1, \dots, n$
			\item $w_i' \longrightarrow^{u^* c} w_i$,
			\item $w_n = w.$
		\end{enumerate}
		We use induction on $n$. If $n = 0$ then by Lemma \ref{lemma:seeding} we see that $w_0 \in L_S$, hence $w_0 \neq_G 1$. 
		
		Now suppose that the result has been established for all words that can be obtained via sequences of length $\leq n-1$. Then $w_{n-1} \neq_G 1$ by the induction hypothesis. Using Lemma \ref{lemma:pulling} we see that $w_{n-1} = \phi_i(w_n')$, where $\phi_i = \phi_L$ if $h_i = h_L$ and $\phi_i = \phi_R$ if $h_i = h_R$, thus $w_n' \neq_G 1$. Using Lemma \ref{lemma:reduction} we see that $w_n' =_G w_n$ and hence $w_n \neq_G 1$, so $w$ indeed represents a nontrivial element of Grigorchuk's group.
	\end{proof}
	
	The next lemma establishes the completeness of our grammar.
	\begin{lem}
		\label{lemma:completeness}
		Let $\{a,b,c,d\}^*$ be arbitrary. If $w \neq_G 1$ then $w \in L$.
	\end{lem}
	\begin{proof}
		Suppose that $w \neq_G 1$. Following the algorithm described in Subsection \ref{sub:wp_algorithm} there is a sequence of words $w_0, w_0', \dots, w_n, w_n'$ and a sequence of maps $\phi_1, \dots, \phi_n \in \{\phi_L, \phi_R\}$ such that
		\begin{enumerate}
			\item $w_i$ is the reduced word obtained by reducing $w_i'$ for $i = 0, \dots, n$,
			\item $|w_i|_a$ is even for $i \geq 1$ and odd for $i = 0$,
			\item $w_{i-1}'$ is obtained from $w_i$ by applying $\phi_i$,
			\item $w_n' = w$.
		\end{enumerate}
			
		Again, we proceed by induction on $n$.
		
		Suppose that $n = 0$, i.e. $w_0' = w$. As the parity of the number of occurrences of the symbol $a$ is invariant with respect to the reduction rules, we see that $|w|_a$ is odd. It follows by Lemma \ref{lemma:seeding} that $S_0 \longrightarrow^{s^*} w$ and thus $w \in L$.
		
		Now suppose that the statement holds for all $w' \in \{a,b,c,d\}^*$ for which the word problem algorithm uses up to $n-1$ levels of recursion. In particular, 
		\begin{displaymath}
			S_0 \longrightarrow^{s^* \{p^*h_L u^*t, p^*h_R u^*t\}^*} w_{n-1}'.
		\end{displaymath}
		Using Lemma \ref{lemma:pulling} we see that $w_{n-1}' \longrightarrow^{p^*h_n} w_n$, where $h_n = h_L$ if $\phi_n = \phi_L$ and  $h_n = h_R$ if $\phi_n = \phi_R$. Similarly, using Lemma \ref{lemma:reduction} we see that $w_n \longrightarrow^{u^* t} w_n'$. Altogether we see that $w_{n-1}' \longrightarrow^{p^*h_n u^* t} w$ and therefore $w \in L$.
	\end{proof}
	
	Combining Lemma \ref{lemma:consistency} and Lemma \ref{lemma:completeness} we see that a word $w \in \{a,b,c,d\}^*$ represents a nontrivial element of Grigorchuk's group if and only if $w \in L$, which implies the main result of this section.
	\begin{thm}
		The co-word problem in Grigorchuk's group is an ETOL language.
	\end{thm}

\section{ET0L languages and $3$-manifold groups}\label{sec:BG}

The goal of this section is to prove Theorem \ref{ET0Lcombing}, which is a strengthening of Theorem B in \cite{MR1420509}, proved there for indexed instead of ET0L languages.
In 1996 Bridson and Gilman stated the theorem for all manifolds satisfying the geometrisation conjecture, but since then Perelman \cite{Perelman} proved that all compact $3$-manifolds do, so we can state the result in full generality. 

\begin{thm}\label{ET0Lcombing}
Let $M$ be a compact $3$-manifold or orbifold, and let $\mu:\Sigma^{\star} \rightarrow \pi_1 M$ be a choice of generators. Then there exists a set of normal forms $L \subseteq \Sigma^{\star}$ which satisfies the asynchronous fellow-traveler property and is an ET0L language.
\end{thm}

We follow in the footsteps of Bridson and Gilman, whose proof relies on (1) closure properties of AFL languages, and (2) showing that an appropriate set of normal forms for $\Z^2$ is ET0L. Note that standard regular normal forms for $\Z^2$ will not produce an appropriate language of normal forms for the extension $\Z^2 \rtimes \Z$, which needs to satisfy the properties detailed in \cite[page 541]{MR1420509}.

For the sake of completeness we recall the relevant results on AFL languages.
As in the paper of Bridson and Gilman (and much of the literature), we will call a set of normal forms satisfying the asynchronous fellow-traveler property a \emph{combing}.

\begin{prop}[\cite{MR1420509}] \label{ET0Lclosure}Let $\mathcal{A}$ be a full AFL class of languages (such as regular, context-free, indexed, or ET0L). 
\begin{enumerate}
\item (Prop. 2.9) If $G_1$ and $G_2$ both have an asynchronous $\mathcal{A}$-combing, then so does the free product $G_1 \ast G_2$.
\item (Theorem 2.16) Let $G$ be a finitely generated group, and $H$ a subgroup of finite index. Then $G$ admits an asynchronous $\mathcal{A}$-combing if and only if $H$ admits an asynchronous $\mathcal{A}$-combing.
\end{enumerate}
\end{prop}

We first recall the crossing sequence $\kappa(m,n)$ of Bridson and Gilman, which gives the EDT0L normal form for $\Z^2$. Let $(m,n) \in \Z^2$ have $m>0, n\geq 0$, and consider the line $l(m,n^+)$ in the plane from $(0,0)$ to $(m,n^+)$, where $n^+$ is chosen slightly larger than $n$, but small enough to ensure that (1) $l(m,n^+)$ does not contain any lattice points except $(0,0)$, and (2) there are no lattice points in the interior of the triangle with vertices $(0,0)$, $(m,n)$ and $(m,n^+)$. For the line $l(m,n^+)$, the sequence formed by recording an $h$ each time a horizontal line in the plane is crossed and a $v$ each time a vertical line is crossed is called the crossing sequence $\kappa(m,n)$. For example, $\kappa(2,3)=hvhhv$.

\begin{thm} \label{Z2}
The set $L=\{\kappa(m,n) \mid m>0, n\geq 0\}$ is an EDT0L language. That is, the indexed combing for $\Z^2$ in \cite{MR1420509} is in fact EDT0L.
\end{thm}

\begin{proof}
The proof only focuses on the first quadrant in $\Z^2$, and it can be easily extended to all of $\Z^2$. As is described in \cite{MR1420509}, $L$ can be generated by starting with an arbitrary $v^k$, $k>0$, and alternately replacing all $v$'s by $h^i v$ and all $h$'s by $v^j h$. The EDT0L grammar thus first has to produce $v^k$, and then apply maps which mimic the morphisms described. Let $\{q, v, h\}$ be the extended alphabet, with $q$ the start symbol. Let $\phi_q$, $\phi_v$, $\phi_h$ and $\phi_s$ be the maps defined by $\phi_q(q)=qv$, $\phi_v(v)=hv$, $\phi_h(h)=vh$ and $\phi_s(q)=v$.

Then $\phi_s\phi_q^{k-1}(q)=v^k$ generates the starting point of the crossing sequence, and then we apply any map in $\{\phi_v, \phi_h\}^{\star}$ to $v^k$ and obtain the set $L$. Thus by Definition \ref{def:edt0lasfeld} the set $L$ is an EDT0L language.
 \end{proof}

We remark that the EDT0L characterisation for the $\Z^2$ combing cannot be lifted to $\Z^2 \rtimes \Z$ and other more general groups, because EDT0L languages do not form a full AFL, which is essential in several proofs in \cite{MR1420509}.

\begin{prop}\label{ET0LNilSol}
Every semidirect product of the form $\Z^2 \rtimes \Z$ admits an asynchronous ET0L combing.
\end{prop} 
\begin{proof}
The proof is exactly as that of Corollary 3.5 in \cite{MR1420509}. More precisely, let $t$ be a generator of $\Z$ and $L$ be the language of normal forms for $\Z^2$ from Theorem \ref{Z2}. By \cite[Theorem 3.1]{MR1420509}, the language $L_0=\{t^* \cup (t^{-1})^*\}L$ is an asynchronous combing of $\Z^2 \rtimes \Z$, and since ET0L languages are closed under concatenation with a regular language and finite unions, we get that $L_0$ is ET0L.
\end{proof}

\begin{proof} (of Theorem \ref{ET0Lcombing})
The work of Thurston, Epstein and Perelman implies that any $\pi_1 M$ as in the hypothesis is commensurable to the free product of an automatic group and (possibly) finite extensions of groups of the form $\Z^2 \rtimes \Z$. Thus the proof follows immediately from Propositions \ref{ET0Lclosure} and \ref{ET0LNilSol}.
\end{proof}

\section{Open problems}

Among the formal languages naturally appearing in group theory none is more prominent than the word problem, that is, the set of words representing the trivial element in a finitely generated group.
Since EDT0L languages are not closed under inverse homomorphism, {\em a priori} a group may have EDT0L word problem for one finite generating set but not for another. However, we do not know of any infinite group which has EDT0L word  or co-word problem  for some finite generating set. Since EDT0L languages are relatively close in complexity to context-free languages, one might wonder whether the groups with context-free word problem have EDT0L word problem. A first negative answer is given below.

\begin{prop}\label{prop:free2notEdt0l}
	Let $F$ be the free group of rank at least two. Then the word problem in $F$ is not EDT0L.
\end{prop}
\begin{proof}
	It was proved in \cite{Latteux} that if a language $L$ is a context-free generator, i.e. for every context-free language $K$ there is a regular language $R_K$ and a homomorphism $h_K$ such that $K = h_K(L \cap R_K)$, then $L$ is not EDT0L. It follows by the Chomsky-Sch\"utzenberger representation theorem \cite{chomsky1963algebraic} that every Dyck language on at least two letters is a context-free generator. It can be easily seen that the word problem in $F_n$, the free group on $n$ generators, is isomorphic to $D^{\star}_n$, the symmetric Dyck Language on $n$ letters. It follows that if $n > 1$ then the word problem in $F_n$ is not EDT0L.
\end{proof}
 
 \begin{question} Is the word problem for $\Z$ (for some or every finite generating set)  EDT0L?
 \end{question}
 
A related open problem is to determine the class of groups having ET0L word problem; a well known problem is to find a non-virtually free group with indexed word problem, so a reasonable conjecture here is that the only groups with ET0L word problem are virtually free.
 
\begin{conjecture}
A group has EDT0L word problem if and only if it is finite. A group has ET0L word problem if and only if it is virtually free.
\end{conjecture}

\section*{Acknowledgments}
We would like to thank Sylvain Salvati for the EDT0L grammar used in the proof of Proposition~\ref{prop:MG}
and Michel Latteux for  the outline for Proposition~\ref{prop:free2notEdt0l}. We also thank Claas R\"over for interesting discussions. Further, we would like to thank the anonymous referee for suggesting several simplifications of the submitted manuscript.

The first two authors were partially supported by the Swiss National Science Foundation grant Professorship FN PP00P2-144681/1, and by a Follow-On grant of the International Centre of Mathematical Sciences in Edinburgh. All authors were supported by the Australian Research Council Discovery Project grant  DP160100486.

\bibliography{refs} \bibliographystyle{plain}

\end{document}